\documentclass[12pt]{article}

\usepackage{amsmath}
\usepackage{mathtools}
\usepackage{amssymb}
\usepackage{fullpage}
\usepackage{color}
\usepackage{enumerate}
\usepackage{cite}
\usepackage{dsfont}

\definecolor{lblue}{RGB}{0,110,152}

\definecolor{dred}{RGB}{171,67,53}

\newtheorem{theorem}{Theorem}%[\arabic{section}]

\newtheorem{proposition}[theorem]{Proposition}

\newtheorem{remark}[theorem]{Remark}

\newtheorem{problem}[theorem]{Problem}

\newcommand{\mendth}{\hfill \ensuremath{\vartriangle}}

\DeclareMathOperator*{\col}{col}

\DeclareMathOperator*{\diag}{diag}

\DeclareMathOperator{\eps}{\varepsilon}

\newenvironment{proof}{{\it Proof :~}}{\hfill$\diamondsuit$\\}

\begin{document}

\title{Simple interval observers for linear impulsive systems with applications to sampled-data and switched systems}

\author{Corentin Briat and Mustafa Khammash\thanks{Corentin Briat and Mustafa Khammash are with the Department of Biosystems Science and Engineering, ETH-Z\"{u}rich, Switzerland; email: mustafa.khammash@bsse.ethz.ch, corentin.briat@bsse.ethz.ch, corentin@briat.info; url: https://www.bsse.ethz.ch/ctsb/, http://www.briat.info.}}

\date{}

\maketitle

%\address{Department of Biosystems Science and Engineering, ETH-Z\"{u}rich, Switzerland.}

%\begin{keyword}
%Impulsive systems; sampled-data systems; switched systems; interval observers
%\end{keyword}

\begin{abstract}
Sufficient conditions for the design of a simple class of interval observers for linear impulsive systems subject to minimum and range dwell-time constraints are obtained and formulated in terms of infinite-dimensional linear programs. The proposed approach is fully constructive in the sense that suitable observer gains can be extracted from the solution of the optimization problems and is flexible enough to be extended to include performance constraints and parametric uncertainties. In order to be solvable, the infinite-dimensional linear programs are relaxed using a method based on sum of squares which is known to be asymptotically exact in the present case. Three examples are given for illustration: the first one pertains on the interval observation of an impulsive system under a minimum dwell-time constraint, the second one is about the interval observation of an aperiodic sampled-data system and the last one is about the interval observation of a linear switched system.
\end{abstract}

\section{Introduction}

%\begin{itemize}
%  \item Can be applied to sampled-data systems \cite{Efimov:16} \cite{Mazenc:14b} and to switched systems. Nothing apparently.
%\end{itemize}

Impulsive linear systems are an important class of hybrid systems that can be used to model a wide variety of real world processes \cite{Goebel:12} and to represent switched and sampled-data systems \cite{Naghshtabrizi:08,Goebel:12,Briat:13d,Briat:14f,Geromel:15}. Among this class of systems, the subclass of linear positive impulsive systems have been recently studied in \cite{Briat:16c} and several stability and stabilization conditions under various dwell-time constraints have been obtained and formulated in terms of, finite- or infinite-dimensional, linear programs which can then be solved using recent optimization techniques. Although restrictive, the class of linear positive impulsive systems can be used as comparison systems for more general impulsive systems or may be helpful in the design of interval observers. Interval observers have been first introduced in \cite{Gouze:00} in order to account for the presence of disturbances acting on the observed system. The key idea is to build two, possibly coupled, observers whose goal will be the real-time estimation of a lower bound and an upper bound for the state of the system. The problem of designing such observers have then been considered for many different classes of systems including systems with inputs \cite{Mazenc:11,Briat:15g}, linear systems \cite{Mazenc:12}, uncertain systems \cite{Bolajraf:15}, delay systems \cite{Efimov:13c,Briat:15g}, impulsive systems \cite{Degue:16nolcos}, LPV systems \cite{Efimov:13b}, discrete-time systems \cite{Mazenc:13,Briat:15g}, systems with samplings \cite{Mazenc:14b,Efimov:16}, etc. %These observers have also been shown to be interesting for stabilization and estimation \cite{Efimov:13b,Efimov:15,Efimov:15b}. Even if not very often mentioned, these observers are, in some sense, very similar to set-valued observers; see e.g. \cite{Shamma:99,Blanchini:12}.

Some classes of interval observers for linear impulsive systems have been proposed in \cite{Degue:16nolcos}. Unfortunately, these reported conditions, which are based on the discrete-time quadratic conditions obtained in \cite{Briat:13d}, only characterize the asymptotic stability of the error dynamics and cannot be used for  design purposes because of their strong non-convex structure. Moreover, this approach does not exploit the positive nature of the error dynamics which may help in the design of suitable observer gains by, for instance, considering diagonal or linear Lyapunov functions. Finally, it is also difficult to guarantee the stability under a desired dwell-time constraint or to ensure a desired performance level as these properties can only be checked a posteriori; i.e. after manually selecting the gains of the interval observer. On the other hand, very few results have been devoted to the interval observation of sampled-data systems \cite{Mazenc:14b,Efimov:16} while, to the author's best knowledge, no results have been reported so far in the context of switched systems.

The main objective of the current paper will be derivation of constructive sufficient conditions for the design of a simple class of interval observers for linear impulsive systems that can guarantee prescribed range and minimum dwell-time constraints. As sampled-data systems and switched systems both admit an impulsive system representation, the developed approach also readily applies to those classes of systems. The overall method relies on the linear programming stability conditions recently obtained in \cite{Briat:16c} for linear positive impulsive/switched systems referred to as clock-dependent conditions because of their dependence of the (relative) time elapsed since the last discrete-event (i.e. the last impulse time or the last switching time). These conditions are the linear programming analogues of the semidefinite programming conditions obtained in \cite{Briat:13d,Briat:14f,Briat:15f} in the context of general linear impulsive, sampled-data and switched systems. These conditions have the benefits of being linear in the system matrices and to readily allow for the derivation of design conditions -- in the current context, for the design of interval observers. The conditions can be checked using discretization methods \cite{Allerhand:11}, using linear programming via the use of Handelman's theorem \cite{Handelman:88,Briat:11h} or using semidefinite programming via the use of Putinar's Positivstellensatz \cite{Putinar:93} combined with computational sum of square methods \cite{Parrilo:00,sostools3}. The proposed approach is flexible enough to incorporate performance constraints (e.g. in the $L_1$ or the $L_\infty$ sense \cite{Briat:11h,Briat:15g}) and uncertain parameters. These extensions, however, are left for future research due to space restrictions.

\noindent\textbf{Outline.} The structure of the paper is as follows: in Section \ref{sec:preliminary} preliminary definitions and results are given. Section \ref{sec:main} is devoted to the derivation of design conditions for the considered class of interval observers. Computational issues are discussed in Section \ref{sec:computational} and examples, finally, are treated in Section \ref{sec:examples}.\\

\noindent\textbf{Notations.} The set of nonnegative integers is denoted by $\mathbb{N}_0$. The cones of positive and nonnegative vectors of dimension $n$ are denoted by $\mathbb{R}_{>0}^n$ and $\mathbb{R}_{\ge0}^n$, respectively. The set of diagonal matrices of dimension $n$ is denoted by $\mathbb{D}^n$ and the subset of those being positive definite is denoted by $\mathbb{D}_{\succ0}^n$. The $n$-dimensional vector of ones is denoted by $\mathds{1}_n$. For some elements, $\{x_1,\ldots,x_n\}$, the operator $\diag_{i=1}^n(x_i)$ builds a matrix with diagonal entries given by $x_1,\ldots,x_n$ whereas $\col_{i=1}^n(x_i)$ creates a vector by vertically stacking them with $x_1$ on the top.

\section{Preliminaries on linear positive impulsive systems}\label{sec:preliminary}

The objective of this section is to recall few results about linear positive impulsive systems \cite{Briat:16c}. To this aim, let us consider the following class of linear impulsive system:
\begin{equation}\label{eq:mainsyst2}
\begin{array}{rcl}
  \dot{x}(t_k+\tau)&=&A(\tau)x(t_k+\tau)+E_c(\tau)w_c(t_k+\tau),\ \tau\in(0,T_k]\\
  x(t_k^+)&=&Jx(t_k)+E_dw_d(k)
\end{array}
\end{equation}
where $k\in\mathbb{N}_0$, $x(t_k^+):=\lim_{s\downarrow t_k}x(s)$ and the matrix-valued functions $A(\tau)\in\mathbb{R}^{n\times n}$ and $E(\tau)\in\mathbb{R}^p$ are continuous and bounded. The sequence of impulse times $\{t_k\}_{k\in\mathbb{N}_0}$, $t_0=0$, is assumed to verify the properties: (a) $T_k:=t_{k+1}-t_k>0$ for all $k\in\mathbb{N}_0$ and (b) $t_k\to\infty$ as $k\to\infty$. When all the above properties hold, the solution of the system \eqref{eq:mainsyst2} exists for all times.

We have the following result regarding the state positivity of the impulsive system \eqref{eq:mainsyst2}.
\begin{proposition}
The following statements are equivalent:
\begin{enumerate}[(a)]
  \item The system \eqref{eq:mainsyst2} is state positive, i.e. for any $x_0\ge0$, $w_c(t)\ge0$ and $w_d(k)\ge0$, we have that $x(t)\ge0$ for all $t\ge0$.
  \item The matrix-valued function $A(\tau)$ is Metzler for all $\tau\ge0$, the matrix-valued function $E_c(\tau)$ is nonnegative for all $\tau\ge0$ and the matrices $J,E_d$ are nonnegative.
\end{enumerate}
\end{proposition}

\subsection{Stability under range dwell-time}

The following result provides a sufficient condition for the stability of the system \eqref{eq:mainsyst2} under a range dwell-time constraint; i.e. $T_k\in[T_{min},T_{max}]$, $k\ge0$, for some given $0<T_{min}\le T_{max}<\infty$. It is an extension of the range dwell-time result derived in \cite{Briat:16c}.
\begin{theorem}\label{th:rangeDT}
Let us consider the system \eqref{eq:mainsyst2} with $w_c\equiv0$, $w_d\equiv0$ and assume that it is state positive; i.e. $A(\tau)$ is Metzler for all $\tau\in[0,T_{max}]$ where $0<T_{min}\le T_{max}<\infty$ are given real numbers. Then, the following statements are equivalent:
\begin{enumerate}[(a)]
  \item There exists a vector $\lambda\in\mathbb{R}_{>0}^n$ such that
  \begin{equation}\label{eq:dksmldksdksmdk}
      \lambda^T\left(J\Phi(\theta)-I_n\right)<0
  \end{equation}
  holds for all $\theta\in[T_{min},T_{max}]$ where
  \begin{equation}
    \dot{\Phi}(s)=A(s)\Phi(s),\ \Phi(0)=I_n, s\in[0,T_{max}].
  \end{equation}
  \item There exist a differentiable vector-valued $\zeta:[0,\bar T]\mapsto\mathbb{R}^n$, $\zeta(0)>0$, and a scalar $\eps>0$ such that the conditions
 \begin{equation}
          \begin{array}{rcl}
            \dot\zeta(\tau)^T+\zeta(\tau)^TA(\tau)&\le&0\\
           \zeta(0)^T J-\zeta(\theta)^T+\eps\mathds{1}^T&\le&0
          \end{array}
        \end{equation}
hold for all $\tau\in[0,\bar T]$ and all $\theta\in[T_{min},T_{max}]$.
\end{enumerate}
Moreover, when one of the above statements holds, then the positive impulsive system \eqref{eq:mainsyst2} is asymptotically stable under range dwell-time $[T_{min},T_{max}]$.\hfill\mendth
\end{theorem}
\begin{proof}
  The proof of this result is similar to the one of the corresponding result for LTI positive system in \cite{Briat:16c}. However, since the system is not time-invariant, we have to use state-transition matrices instead of matrix exponentials to prove the equivalence as in \cite{Briat:15f}. Due to space limitations, this result is not proved here.
\end{proof}

The condition stated in the statement (b) in the above result forms an infinite-dimensional linear program that cannot be solved per se. However, relaxed conditions taking the form of finite-dimensional semidefinite/linear programs can be obtained by restricting $\zeta(\tau)$ to be a vector-valued polynomial. Note that in this case, the derivative $\dot\zeta(\tau)$, which is also polynomial, is immediate to get by differentiation. This procedure will be explained in more details in Section \ref{sec:computational}.

%\begin{remark}[Constant dwell-time case]
%When $T_{min}=T_{max}$, then the conditions in the above result reduces to the constant dwell-time result of \cite{Briat:16c}. Moreover, the conditions interestingly become necessary and sufficient for the asymptotic stability of the system. This follows from the fact that when $T_{min}=T_{max}$, then the condition \eqref{eq:dksmldksdksmdk} is equivalent to the Schur stability of the matrix $Je^{A\theta}$, which is equivalent in turn to the asymptotic stability of the impulsive system with periodic impulses.
%\end{remark}

\subsection{Stability under minimum dwell-time}

The following result provides a sufficient condition for the stability of the system \eqref{eq:mainsyst2} under a minimum dwell-time constraint; i.e. $T_k\ge\bar T$, $k\ge0$, for some given $\bar T>0$. It is an extension of the minimum dwell-time result derived in \cite{Briat:16c}:
\begin{theorem}\label{th:minDT}
Let us consider the system  \eqref{eq:mainsyst2} with $w_c\equiv0$, $w_d\equiv0$, $A(\tau)=A(\bar{T})$ for all $\tau\ge\bar{T}>0$, where $\bar T>0$ is given, and assume that it is state positive. Then, the following statements are equivalent:
\begin{enumerate}[(a)]
  \item There exists a vector $\lambda\in\mathbb{R}_{>0}^n$ such that
  \begin{equation}
      \lambda^TA(\bar{T})<0
  \end{equation}
  and
    \begin{equation}
      \lambda^T\left(J\Phi(\bar{T})-I_n\right)<0
  \end{equation}
  hold  where
  \begin{equation}
    \dot{\Phi}(s)=A(s)\Phi(s),\ \Phi(0)=I_n, s\in[0,\bar{T}].
  \end{equation}
  \item There exist a differentiable vector-valued $\zeta:[0,\bar T]\mapsto\mathbb{R}^n$, $\zeta(\bar T)>0$, and a scalar $\eps>0$ such that the conditions
 \begin{equation}
          \begin{array}{rcl}
           \zeta(\bar T)^TA(\bar{T})&<&0\\
            -\dot\zeta(\tau)^T+\zeta(\tau)^TA(\tau)&\le&0\\
           \zeta(\bar T)^T J-\zeta(0)^T+\eps\mathds{1}^T&\le&0
          \end{array}
        \end{equation}
hold for all $\tau\in[0,\bar T]$.
\end{enumerate}
Moreover, when one of the above statements holds, then the positive impulsive system \eqref{eq:mainsyst2} is asymptotically stable under minimum dwell-time $\bar T$.\hfill\mendth
\end{theorem}
\begin{proof}
For the same reasons as for Theorem \ref{th:1}, the proof is omitted.
  %The proof of this result is similar to the one of the corresponding result for LTI positive system in \cite{Briat:16c}. However, since the system is not time-invariant, we have to use state-transition matrices instead of matrix exponentials to prove the equivalence as in \cite{Briat:15f}. Due to space limitations, this result is not proved here.
\end{proof}

\subsection{Extensions to systems with inputs}

It seems important to clarify the fact that under the assumption of the asymptotic stability of the system without input, we have that the state of the system remains bounded provided that the inputs are bounded. Such a result is stated below where only the range dwell-time case is considered for brevity:
%
%
%
%\begin{equation}\label{eq:mainsyst2w}
%\begin{array}{rcl}
%  \dot{x}(t_k+\tau)&=&A(\tau)x(t_k+\tau)+E_c(\tau)w(t_k+\tau),\ \tau\in(0,T_k]\\
%  x(t_k^+)&=&Jx(t_k)+E_dw_d(k).
%\end{array}
%\end{equation}
%
%This is stated in the following proposition:
\begin{proposition}\label{prop:input}
  Assume that the linear impulsive system \eqref{eq:mainsyst2} is state positive and that the conditions of Theorem \ref{th:1} hold for some $0<T_{min}\le T_{max}<\infty$. Then, the system \eqref{eq:mainsyst2}  is input-to-state stable under range dwell-time $[T_{min},T_{max}]$.
\end{proposition}
\begin{proof}
Assume that the conditions of statement (a) of the result on range dwell-time (Theorem \ref{th:1}) are met. Then, there exists an $\epsilon\in(0,1)$ such that $\lambda^T\left(Je^{AT_k}-I_n\right)\le(\epsilon-1)\lambda^T$ for all $T_k\in[T_{min},T_{max}]$. Letting then $V(x)=\lambda^Tx$ and $\Psi(t,s)=\Phi(t)\Phi(s)^{-1}$, we get from \eqref{eq:mainsyst} that
\begin{equation}
\begin{array}{rcl}
    V(t_{k+1}^+)-V(t_{k}^+)&=&\lambda^T\left(J\Phi(T_k)-I_n\right)x(t_k^+)\\
    &&+\lambda^TE_dw_d(k+1)\\
    &&+\lambda^T\textstyle\int_0^{T_k}J\Psi(T_k,s)E_c(s)w(t_k+s)ds\\
    &\le&(\epsilon-1) V(t_{k}^+)+\mu
\end{array}
\end{equation}
where $\mu:=\lambda^T\left(E_d\mathds{1}||w_d||_{\ell_\infty}+JM\mathds{1}||w_c||_{L_\infty}\right)$ and $M$ is such that
\begin{equation}
\int_0^{T_{k}}\Psi(T_k,s)E_c(s)ds\le M
\end{equation}
for all $T_k\in[T_{min},T_{max}]$. This then leads to
\begin{equation}
  V(t_k^+)\le(1-\epsilon)^kV(0)+\dfrac{1-\epsilon^k}{1-\epsilon}\mu
\end{equation}
and hence that
\begin{equation}
  \limsup_{k\to\infty}V(t_k^+)\le\dfrac{\mu}{1-\epsilon}<\infty.
\end{equation}
The boundedness of $V(x((t))$ for all $t\ge0$, then simply follows from an application of the Bellman-Gr\"{o}nwall lemma. The proof is completed.
%
%To show that boundedness holds for all $t\ge0$, first note that
%\begin{equation}
%\begin{array}{rcl}
%    V(x(t_k+\tau))&=&V(x(t_k^+))+\int_0^\tau\lambda^TA(s)x(t_k+s)ds\\
%    &&+\int_0^\tau\lambda^TE_c(s)w_c(t_k+s)ds\\
%                            &\le&\alpha_k+\beta\int_0^\tau V(x(t_k+s))ds
%\end{array}
%\end{equation}
%where $\alpha_k:=V(x(t_k^+))+\lambda^T\int_0^{T_{max}}E_c(s)\mathds{1}_{p_c}ds||w_c||_{\ell_\infty}$ and a scalar $\beta\ge0$ is such that $\lambda^TA_c(s)\le\beta\lambda^T$ for all $s\in[0,T_{max}]$. Then, from Bellman-Gr\"{o}nwall lemma, we have that
%\begin{equation}
%    V(x(t_k+\tau))\le\alpha_ke^{\beta(t-t_k)}\le \bar{\alpha}e^{\beta T_{max}}
%\end{equation}
%where
%\begin{equation}
%\begin{array}{rcl}
%    \bar{\alpha}&:=&\limsup_{k\to0}\alpha_k\le\dfrac{\mu}{1-\epsilon}\\
%    &&+\lambda^T\int_0^{T_{max}}E_c(s)\mathds{1}_{p_c}ds||w_c||_{\ell_\infty}
%\end{array}
%\end{equation}
%and hence we have that
%\begin{equation}
%  \limsup_{k\to\infty}V(x(t_k+\tau))\le\bar{\alpha}e^{\beta T_{max}}
%\end{equation}
%which shows the boundedness of the trajectories of the system. The proof in the ranged dwell-time case is completed.
\end{proof}

\section{Main results}\label{sec:main}

%\subsection{Definition of the system}

Let us consider the following class of linear impulsive system
\begin{equation}\label{eq:mainsyst}
\begin{array}{rcl}
  \dot{x}(t)&=&Ax(t)+E_cw_c(t),\ t\notin\{t_k\}_{k\in\mathbb{N}}\\
  x(t_k^+)&=&Jx(t_k)+E_dw_d(k),\ k\in\mathbb{N}\\
  y_c(t)&=&C_{c}x(t)+F_cw_c(t),\ t\notin\{t_k\}_{k\in\mathbb{N}}\\
  y_d(k)&=&C_{d}x(t_k)+F_dw_d(k),\ k\in\mathbb{N}\\
  x(t_0)&=&x_0,\ t_0=0
\end{array}
\end{equation}
where $x,x_0\in\mathbb{R}_{\ge0}^n$, $w_c\in\mathbb{R}_{\ge0}^{p_c}$, $w_d\in\mathbb{R}_{\ge0}^{p_d}$, $y_c\in\mathbb{R}_{\ge0}^{q_c}$ and $y_d\in\mathbb{R}_{\ge0}^{q_d}$ are the state of the system, the initial condition, the continuous-time exogenous input, the discrete-time exogenous input, the continuous-time measured output and the discrete-time measured output, respectively. The sequence of impulse instants $\{t_k\}_{k\in\mathbb{Z}_{\ge0}}$ is assumed to satisfy the same properties as for the system \eqref{eq:mainsyst2}. The input signals are all assumed to be bounded functions and that some bounds are known; i.e. we have $w_c^-(t)\le w_c(t)\le w_c^+(t)$ and $w_d^-(k)\le w_d(k)\le w_d^+(k)$ for all $t\ge0$ and $k\ge0$ and for some known $w_c^-(t), w_c^+(t),w_d^-(k),w_d^+(k)$.

\subsection{Proposed interval observer}

We are interested in finding an interval-observer of the form
\begin{equation}\label{eq:obs}
\begin{array}{lcl}
      \dot{x}^\bullet(t)&=&Ax^\bullet(t)+E_c w_c^\bullet(t)\\
&&\quad+L_c(t)(y_c(t)-C_c x^\bullet(t)-F_c w_c^\bullet(t))\\
      x^\bullet(t_k^+)&=&Jx^\bullet(t_k)+E_d w_d^\bullet(t)\\
&&\quad+L_d(y_d(k)-C_d x^\bullet(t_k)-F_d w_d^\bullet(t))\\
      x^\bullet(0)&=&x_0^\bullet
\end{array}
\end{equation}
where $\bullet\in\{-,+\}$. Above, the observer with the superscript ``$+$'' is meant to estimate an upper-bound on the state value whereas the observer with the superscript ``-'' is meant to estimate a lower-bound, i.e. $x^-(t)\le x(t)\le x^+(t)$ for all $t\ge0$ provided that $x_0^-\le x_0\le x_0^+$, $w_c^-(t)\le w_c(t)\le w_c^+(t)$ and $w_d^-(k)\le w_d(k)\le w_d^+(k)$.

The errors dynamics $e^+(t):=x^+(t)-x(t)$ and $e^-(t):=x(t)-x^-(t)$ are then described by
\begin{equation}\label{eq:error}
\begin{array}{rcl}
    \dot{e}^\bullet(t)&=&(A-L_c(t)C_c)e^\bullet(t)+(E_c-L_c(t) F_c)\delta_c^\bullet(t)\\ %(E_c-L^\bullet_cF_c)\delta_c^\bullet(t)\\
    e^\bullet(t_k^+)&=&(J-L_d C_d)e^\bullet(t_k)+(E_d-L_d F_d)\delta_d^\bullet(k)\\ %(E_d-L^\bullet_dF_d)\delta_d^\bullet(k)\\
%    \xi^\bullet(t)&=&M^\bullet e^\bullet
%    \dot{e}^-(t)&=&(A-LC)e^+(t)+(E-LF)\delta^-(t)
\end{array}
\end{equation}
where $\bullet\in\{-,+\}$, $\delta_c^+(t):=w_c^+(t)-w_c(t)\in\mathbb{R}_{\ge0}^{p_c}$, $\delta_c^-(t):=w_c(t)-w_c^-(t)\in\mathbb{R}_{\ge0}^{p_c}$, $\delta_d^+(k):=w_d^+(k)-w_d(k)\in\mathbb{R}_{\ge0}^{p_d}$ and $\delta_d^-(k):=w_d(k)-w_d^-(k)\in\mathbb{R}_{\ge0}^{p_d}$. Note that both errors have exactly the same dynamics and, consequently, it is unnecessary here to consider different observer gains. Note that this would not be the case if the observers were coupled in a non-symmetric way.

%The matrices $M_{\sigma(t)}^-,M_{\sigma(t)}^+\in\mathbb{R}^{q\times n}_{\ge0}$ are the nonzero matrices driving the errors $e^\bullet$ to the observed outputs $\xi^\bullet$, that we assume to be chosen a priori.

%\begin{remark}
%  As it may be sometimes difficult to find a matrix $L$ that makes the matrix $A-LC$ both Metzler and Hurwitz stable, and the other matrices nonnegative. To overcome the first problem, the determination of  an invertible matrix $P$ and a matrix $L$ such that $P(A-LC)P^{-1}$ is Metzler and Hurwitz stable can be considered \cite{Raissi:12}. For the second problem other observer structures can be used; see e.g. \cite{Briat:15g} and references therein.
%\end{remark}

\subsection{Range dwell-time result}

In the range-dwell -time case, the time-varying gain $L_c(t)$ in \eqref{eq:obs} is defined as follows
\begin{equation}\label{eq:L1}
  L_c(t)=\tilde{L}_c(t-t_k),\ t\in(t_k,t_{k+1}]
\end{equation}
where $\tilde{L}_c:[0,T_{max}]\mapsto\mathbb{R}^{n\times q_c}$ is a matrix-valued function to be determined. The rationale for considering such structure is to allow for the derivation of convex design conditions. The observation problem is defined, in this case, as follows:
\begin{problem}\label{problem1}
Find an interval observer of the form \eqref{eq:obs} (i.e. a matrix-valued function $L_c(\cdot)$ of the form \eqref{eq:L1} and a matrix $L_d\in\mathbb{R}^{n\times q_d}$) such that the error dynamics \eqref{eq:error} is
  \begin{enumerate}[(a)]
    \item state-positive, that is
    \begin{itemize}
      \item $A-L_c(\tau)  C_c$ is Metzler for all $\tau\in[0,T_{max}]$,
      \item $E_c-L_c(\tau)  F_c $ is nonnegative for all $\tau\in[0,T_{max}]$,
      \item $J-L_dC_d$ and $E_d-L_dF_d$ are nonnegative; and
    \end{itemize}
    \item asymptotically stable under range dwell-time\\ $[T_{min},T_{max}]$ when $w_c\equiv0$ and $w_d\equiv0$.
  \end{enumerate}
\end{problem}
Note that by virtue of Proposition \ref{prop:input}, the property (b) implies that the error dynamics (15) will be stable and positive for any bounded  $w_c$ and $w_d$ satisfying the conditions below \eqref{eq:error}.

The following result provides a sufficient condition for the solvability of Problem \ref{problem1}:
\begin{theorem}\label{th:1}
Assume that there exist a differentiable matrix-valued function $X:[0,T_{max}]\mapsto\mathbb{D}^n$, $X(0)\succ0$, a matrix-valued function $U_c:[0,T_{max}]\mapsto\mathbb{R}^{n\times q_c}$, a matrix $U_d\in\mathbb{R}^{n\times q_d}$ and scalars $\eps,\alpha>0$ such that the conditions
\begin{subequations}\label{eq:th1a}
\begin{alignat}{4}
            X(\tau)A-U_c(\tau)C_c+\alpha I_n&\ge0\label{eq:th1:1}\\
           X(0) J-U_d C_d&\ge0\label{eq:th1:2}\\
           X(\tau)E_c-U_c(\tau)F_c&\ge0\label{eq:th1:3}\\
            X(0)E_d-U_d F_d&\ge0\label{eq:th1:4}
  \end{alignat}
\end{subequations}
           and
 \begin{subequations}\label{eq:th1b}
\begin{alignat}{4}
           \mathds{1}^T_n\left[\dot{X}(\tau)+X(\tau)A-U_c(\tau)C_c\right]&\le0\label{eq:th1:5}\\
            \mathds{1}^T_n\left[X(0) J-U_d C_d-X(\theta)+\eps I\right]&\le0\label{eq:th1:6}
             \end{alignat}
\end{subequations}
%
% \begin{equation}
%          \begin{array}{rcl}
%
%          \end{array}
%        \end{equation}
        %
hold for all $\tau\in[0,T_{max}]$ and all $\theta\in[T_{min},T_{max}]$. Then, there exists an interval observer of the form \eqref{eq:obs}-\eqref{eq:L1} that solves Problem \ref{problem1} and suitable observer gains are given by
\begin{equation}\label{eq:formula1}
  \tilde{L}_c(\tau)= X(\tau)^{-1}U_c(\tau)\quad \textnormal{and}\quad L_d=X(0)^{-1}U_d.
\end{equation}
\end{theorem}
\begin{proof}
From the diagonal structure of the matrix-valued function $X(\cdot)$ and the changes of variables \eqref{eq:formula1}, the inequalities \eqref{eq:th1:1} to \eqref{eq:th1:4} are readily equivalent to saying that the statement (a) of Problem \ref{problem1} holds. Using now the changes of variables $\lambda(\tau)=X(\tau)\mathds{1}_n$ and \eqref{eq:formula1}, we get that the feasibility of \eqref{eq:th1:5}-\eqref{eq:th1:6} is equivalent to saying that the error dynamics \eqref{eq:error} with \eqref{eq:L1} verifies the range dwell-time conditions of Theorem \ref{th:rangeDT} with the same $\lambda(\tau)$.
\end{proof}

\subsection{Minimum dwell-time result}

In the minimum dwell-time case, the time-varying gain  $L_c$ is defined as follows
\begin{equation}\label{eq:L2}
  L_c(t)=\left\{\begin{array}{ll}
  \tilde{L}_c(t-t_k)& \textnormal{if }t\in(t_k,t_{k}+\tau]\\
  \tilde{L}_c(\bar T)& \textnormal{if }t\in(t_k+\bar{T},t_{k+1}]
  \end{array}\right.
\end{equation}
where $\tilde{L}_c:\mathbb{R}_{\ge0}\mapsto\mathbb{R}^{n\times q_c}$ is a function to be determined. As in the range dwell-time case, the structure is chosen to facilitate the derivation of convex design conditions. The observation problem is defined, in this case, as follows:
\begin{problem}\label{problem2}
Find an interval observer of the form \eqref{eq:obs} (i.e. a matrix-valued function $L_c(\cdot)$ of the form \eqref{eq:L2} and a matrix $L_d\in\mathbb{R}^{n\times q_d}$) such that the error dynamics \eqref{eq:error} is
  \begin{enumerate}[(a)]
    \item state-positive, that is
    \begin{itemize}
      \item $A-L_c(\tau) C_c$ is Metzler for all $\tau\in[0,\bar{T}]$,
      \item $E_c-L_c(\tau) F_c $ is nonnegative for all $\tau\in[0,\bar{T}]$,
      \item $J-L_dC_d$ and $E_d-L_dF_d$ are nonnegative; and
    \end{itemize}
    \item asymptotically stable under minimum dwell-time $\bar T$ when $w_c\equiv0$ and $w_d\equiv0$.
  \end{enumerate}
\end{problem}

The following result provides a sufficient condition for the solvability of Problem \ref{problem2}:
\begin{theorem}\label{th:2}
There exists a differentiable matrix-valued function $X:[0,\bar T]\mapsto\mathbb{D}^n$, $X(\bar T)\succ0$, a matrix-valued function $U_c:[0,\bar T]\mapsto\mathbb{R}^{n\times q_c}$, a matrix $U_d\in\mathbb{R}^{n\times q_d}$ and scalars $\eps,\alpha>0$ such that the conditions
\begin{subequations}\label{eq:th2a}
\begin{alignat}{4}
           X(\tau)A-U_c(\tau)C_c+\alpha I_n&\ge0\label{eq:mp1}\\
           X(\bar T) J-U_d C_d&\ge0\label{eq:mp2}\\
           X(\tau)E_c-U_c(\tau)F_c&\ge0\label{eq:mp2}\\
           X(\bar T)E_d-U_d F_d&\ge0\label{eq:mp2}
  \end{alignat}
\end{subequations}
and
    \begin{subequations}\label{eq:th2b}
\begin{alignat}{4}
           \mathds{1}^T_n\left[X(\bar T)A-U_c(\bar T)C_c+\eps I_n\right]&\le0\label{eq:md1}\\
           \mathds{1}^T_n\left[-\dot{X}(\tau)+X(\tau)A-U_c(\tau)C_c\right]&\le0\label{eq:md2}\\
            \mathds{1}^T_n\left[X(\bar T) J-U_d C_d-X(0)+\eps I\right]&\le0\label{eq:md3}
\end{alignat}
\end{subequations}
hold for all $\tau\in[0,\bar T]$.  Then, there exists an interval observer of the form \eqref{eq:obs}-\eqref{eq:L2} that solves Problem \ref{problem2} and suitable observer gains are given by
\begin{equation}
  \tilde{L}_c(\tau)= X(\tau)^{-1}U_c(\tau)\quad \textnormal{and}\quad L_d=X(\bar{T})^{-1}U_d.
\end{equation}
\end{theorem}
\begin{proof}
The proof is similar to the one of Theorem \ref{th:1} and is omitted.
\end{proof}

\section{Computational considerations}\label{sec:computational}

Several methods can be used to check the conditions stated in Theorem \ref{th:1} and Theorem \ref{th:2}. %The piecewise linear discretization approach \cite{Allerhand:11,Xiang:15a,Briat:16c} assumes that the decision variables are piecewise linear functions of their arguments and leads to a finite-dimensional linear program that can be checked using standard linear programming algorithms. Another possible approach is based on Handelman's Theorem \cite{Handelman:88} and also leads to a finite-dimensional program \cite{Briat:11h,Briat:16c}.
We opt here for an approach based on Putinar's Positivstellensatz \cite{Putinar:93} and semidefinite programming \cite{Parrilo:00}\footnote{See \cite{Briat:16c} for a comparison of all these methods.}. Before stating the main result of the section, we need to define first some terminology. A multivariate polynomial $p(x)$ is said to be a sum-of-squares (SOS) polynomial if it can be written as $\textstyle p(x)=\sum_{i}q_i(x)^2$ for some polynomials $q_i(x)$. A polynomial matrix $p(x)\in\mathbb{R}^{n\times m}$ is said to \emph{componentwise sum-of-squares} (CSOS) if each of its entries is an SOS polynomial. Checking whether a polynomial is SOS can be exactly cast as a semidefinite program \cite{Parrilo:00} that can be easily solved using semidefinite programming solvers such as SeDuMi \cite{Sturm:01a}. The package SOSTOOLS \cite{sostools3} can be used to formulate and solve SOS programs in a convenient way.

Below is the SOS implementation of the conditions of statement (b) of Theorem \ref{th:1}:
\begin{proposition}\label{prop:SOS1}
  Let $d\in\mathbb{N}$, $\eps>0$ and $\epsilon>0$ be given and assume that there exist polynomials $\chi_i:\mathbb{R}\mapsto\mathbb{R}$, $i=1,\ldots,n$, $U_c:\mathbb{R}\mapsto\mathbb{R}^{n\times q_c}$, $\Gamma_1:\mathbb{R}\mapsto\mathbb{R}^{n\times n}$,  $\Gamma_2:\mathbb{R}\mapsto\mathbb{R}^{n\times q_c}$ and $\gamma_1,\gamma_2:\mathbb{R}\mapsto\mathbb{R}^{n}$ of degree $2d$, a matrix $U_d\in\mathbb{R}^{n\times q_d}$ and a scalar $\alpha\ge0$ such that
  \begin{enumerate}[(a)]%[(i)]
    \item $\Gamma_i(\tau)$, $\gamma_i(\tau)$, $i=1,2$ are CSOS,
    \item $X(0)-\epsilon I_n\ge0$ (or is CSOS),
    \item $X(\tau)A-U_c(\tau)C_c+\alpha I_n-\Gamma_1(\tau)f(\tau)$ is CSOS,
    \item $X(0) J-U_d C_d\ge0$ (or is CSOS),
    \item $X(\tau)E_c-U_c(\tau)F_c-\Gamma_2(\tau)f(\tau)$ is CSOS,
    \item $X(0)E_d-U_d F_d\ge0$ (or is CSOS),
    \item $-\mathds{1}^T_n\left[\dot{X}(\tau)+X(\tau)A-U_c(\tau)C_c\right]-f(\tau)\gamma_1(\tau)^T$

    is CSOS,
    \item $-\mathds{1}^T_n\left[X(0) J-U_d C_d-X(\theta)+\eps I\right]-g(\theta)\gamma_2(\theta)^T$

     is CSOS,
  \end{enumerate}
  where $X(\tau):=\diag_{i=1}^n(\chi_i(\tau))$, $f(\tau):=\tau(T_{max}-\tau)$ and $g(\theta):=(\theta-T_{min})(T_{max}-\theta)$.

  Then, the conditions of statement (b) of Theorem \ref{th:1} hold with the same $X(\tau)$, $U_c(\tau)$, $U_d$, $\alpha$ and $\eps$.% and the system \eqref{eq:mainsyst} is asymptotically stable under range dwell-time $[T_{min},T_{max}]$.
\end{proposition}
\begin{proof}
The proof follows from the same arguments as the proof of Proposition 3.15 in \cite{Briat:16c}.
%  Clearly, the condition in (ii) is equivalent to saying that $\zeta(\bar T)>0$ and that in (iv) is equivalent to the condition \eqref{eq:c2b2}. We then consider the condition in (iii), which implies that $\zeta(\tau)^\T A-\dot{\zeta}(\tau)\le-\tau(\bar{T}-\tau)\gamma(\tau)^\T$ for all $\tau\in\mathbb{R}$. Since $\gamma(\tau)$ is CSOS, then we have that $\gamma(\tau)\ge0$ for all $\tau\in\mathbb{R}$ and, therefore, this implies that $\zeta(\tau)^\T A-\dot{\zeta}(\tau)\le0$ for all $\tau\in[0,\bar T]$ which is exactly the condition \eqref{eq:c1b2}. The proof is complete.
\end{proof}

The conditions stated in the above results can be readily implemented using SOSTOOLS \cite{sostools3}. However, it seems important to explain the meaning of those conditions. Clearly, the condition (a) implies that $\Gamma_i(\tau)\ge0$, $\gamma_i(\tau)\ge0$, for all $i=1,2$ and all $\tau\in\mathbb{R}$. The conditions (b), (d) and (f) are here to indicate that the corresponding expressions are nonnegative and are equivalent to the conditions $X(0)\succ0$, \eqref{eq:th1:2} and \eqref{eq:th1:4}, respectively. The condition (c) implies that  $X(\tau)A-U_c(\tau)C_c+\alpha I_n-\Gamma_1(\tau)f(\tau)\ge0$ for all $\tau\in\mathbb{R}$. This equivalent to saying that $X(\tau)A-U_c(\tau)C_c+\alpha I_n\ge \Gamma_1(\tau)f(\tau)$ for all $\tau\in\mathbb{R}$ and, hence, that $X(\tau)A-U_c(\tau)C_c+\alpha I_n\ge0$ for all $\tau\in[0,T_{max}]$, which coincides with the condition \eqref{eq:th1:1}. The other conditions can analyzed in the same way.

\begin{remark}[Asymptotic exactness]
The above relaxation is asymptotically exact under very mild conditions \cite{Putinar:93} in the sense that if the original conditions of Theorem \ref{th:1} hold then we can find a degree $d$ for which the conditions in Proposition \ref{prop:SOS1} are feasible. See \cite{Briat:16c} for more details.
\end{remark}

\section{Examples}\label{sec:examples}

\subsection{An impulsive system}

Let us consider here the example from \cite{Briat:13d} to which we add disturbances as also done in \cite{Degue:16nolcos}. The matrices of the system are given by
\begin{equation}\label{eq:ex1}
\begin{array}{l}
    A=\begin{bmatrix}
    -1 & 0\\
    1 & -2
  \end{bmatrix}, E_c=\begin{bmatrix}
    0.1\\
    0.1
  \end{bmatrix}, J=\begin{bmatrix}
2 & 1\\
1 & 3
  \end{bmatrix},  E_d=\begin{bmatrix}
    0.3\\
    0.3
  \end{bmatrix},\\
    C_c=C_d=\begin{bmatrix}
    0 & 1
  \end{bmatrix}, F_c=F_d=0.03.
\end{array}
\end{equation}
Define also $w_c(t)=\sin(t)$, $w^-(t)=-1$, $w^+(t)=1$, $w_d(k)$ is a stationary random process that follows the uniform distribution $\mathcal{U}(-0.5,0.5)$, $w_d^-=-0.5$ and $w_d^+=0.5$. Letting a desired minimum dwell-time of $\bar T=0.7$ and solving the conditions of Theorem \ref{th:2} with polynomials of degree 4 yield the observer gains
\begin{equation}\label{eq:Lex1}
  L_d=\begin{bmatrix}
    0.9977\\
    1.6460
  \end{bmatrix}\quad \textnormal{and}\quad L_c(\tau)=\begin{bmatrix}
    n_1(\tau)d_1(\tau)^{-1}\\
    n_2(\tau)d_2(\tau)^{-1}
  \end{bmatrix}
\end{equation}
where
\begin{equation}
  \begin{array}{rcl}
    n_1(\tau)&=&0.3064\tau^4   -0.4410\tau^3+    0.2132\tau^2\\
    &&   -0.0409\tau+    0.0043\\
    d_1(\tau)&=&0.1999\tau^4   -0.0447\tau^3   -1.0739\tau^2\\
    &&    +2.6471\tau   -2.7157\\
    n_2(\tau)&=&-0.5400 \tau^4   +0.6047\tau^3   -0.0939\tau^2\\
    &&   -0.0771\tau+    0.0251\\
    d_2(\tau)&=&0.0633\tau^4   -0.1169\tau^3+    0.0868\tau^2 \\
    &&  -0.0359 \tau^1+   0.0101.
  \end{array}
\end{equation}
For information, the semidefinite program has 242 primal variables, 76 dual variables and it takes 2.18 seconds to solve on an i7-2620M with 8GB of RAM. To illustrate this result, we generate random impulse times satisfying the minimum dwell-time condition and we obtain the trajectories depicted in Fig.~\ref{fig:states} where we can observe the ability of the interval observer to properly frame the trajectory of the system.

\begin{figure}
  \centering
  \includegraphics[width=0.75\textwidth]{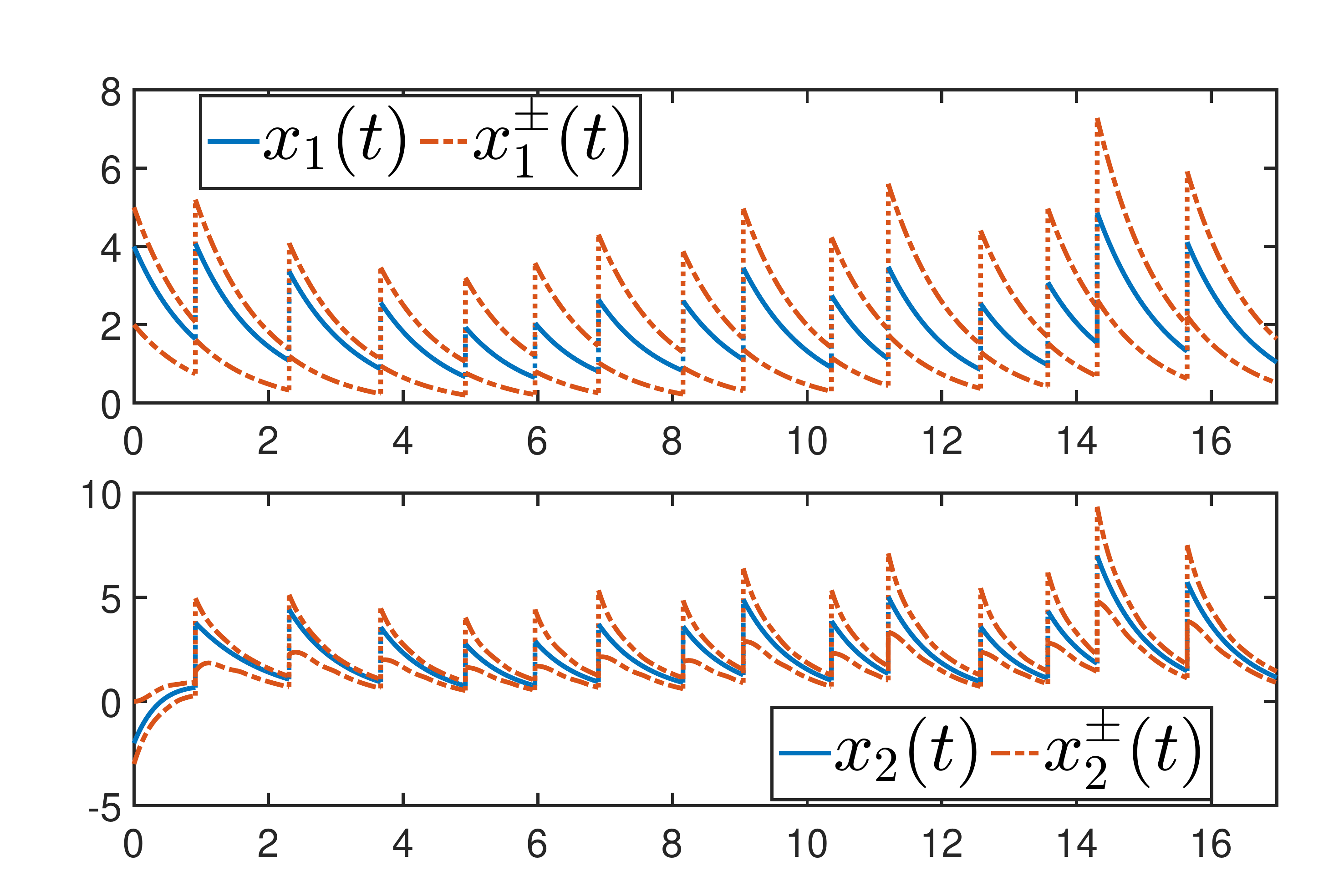}
  \caption{Trajectories of the system \eqref{eq:mainsyst}-\eqref{eq:ex1} and the interval observer \eqref{eq:obs}-\eqref{eq:L2}-\eqref{eq:Lex1} for some randomly chosen impulse times satisfying the minimum dwell-time $\bar{T}=0.7$.}\label{fig:states}
\end{figure}

\subsection{A sampled-data system}

Let us consider now the sampled-data system
\begin{equation}
\begin{array}{rcl}
  \begin{bmatrix}
    \dot{x}(t)\\
    \dot{u}(t)
  \end{bmatrix}&=&\begin{bmatrix}
    A & B\\
    0 & 0
  \end{bmatrix}\begin{bmatrix}
    x(t)\\
    u(t)
  \end{bmatrix}+\begin{bmatrix}
    E\\
    0
  \end{bmatrix}w(t)\\
  \begin{bmatrix}
    x(t_k^+)\\
    u(t_k^+)
  \end{bmatrix}&=&\begin{bmatrix}
    I & 0\\
    K_1C_y & K_2
  \end{bmatrix}\begin{bmatrix}
    x(t_k)\\
    u(t_k)
  \end{bmatrix}+\begin{bmatrix}
    0\\
    K_1F_y
  \end{bmatrix}w(t_k)
  \end{array}
\end{equation}
which incorporates in its formulation the sampled-data static-output feedback control law $u(t)=K_1(C_yx(t_k)+F_yw(t_k))+K_2u(t_k)$, $t\in(t_k,t_{k+1}]$. The goal would be the design of an impulsive interval observer for this system and, to this aim, we consider the measured output $y=\diag(C_yx+F_yw_c,u)$. Note that we have here $y_c(t)=y(t)$, $y_d(k)=y(t_k)$, $w_c(t)=w(t)$ and $w_d(k)=w(t_k)$. We then propose the observer \eqref{eq:obs} with the matrices
\begin{equation}
  L_c(t)=\begin{bmatrix}
    L_c^1(t) & B\\
    0 & L_c^2(t)
  \end{bmatrix}\ \textnormal{and}\   L_d=\begin{bmatrix}
    L_d^1 & 0\\
    K_1 & K_2
  \end{bmatrix}.
\end{equation}
Note that this observer contains a continuous-time component which may be contradictory with the fact we are considering a sampled-data system. However, if the observer is sampled at a much higher frequency than the controller, this approximation is, in general, satisfying. The dynamics of the observation error is given, in this case, by
\begin{equation}\label{eq:error_sampled}
\begin{array}{rcl}
  \begin{bmatrix}
    \dot{e}_x^\bullet(t)\\
    \dot{e}_u^\bullet(t)
  \end{bmatrix}&=&\begin{bmatrix}
    A-L_c^1(t)C_y & 0\\
    0 & -L_c^2(t)
  \end{bmatrix}\begin{bmatrix}
    e_x^\bullet(t)\\
    e_u^\bullet(t)
  \end{bmatrix}\\
  &&+\begin{bmatrix}
    E-L_c^1(t)F_y\\
    0
  \end{bmatrix}\delta_c^\bullet(t)\\
  \begin{bmatrix}
    e_x^\bullet(t_k^+)\\
    e_u^\bullet(t_k^+)
  \end{bmatrix}&=&\begin{bmatrix}
    I-L_d^1C_y & 0\\
    0 & 0
  \end{bmatrix}\begin{bmatrix}
    e_x^\bullet(t_k)\\
    e_u^\bullet(t_k)
  \end{bmatrix}
  \end{array}
\end{equation}
where we can see that with this observer gains, the dynamics of the errors are fully decoupled. Consequently, it is enough to choose $L_c^2(t)$ to be constant, diagonal and large enough. The gains $L_c^1(t)$ and $L_d^1$ can then be designed exactly in the same way as in the previous example.

\subsection{A switched system}

Let us consider here the switched system
\begin{equation}
\begin{array}{rcl}
    \dot{\tilde{x}}(t)&=&\tilde{A}_{\sigma(t)}\tilde{x}(t)+\tilde{E}_{\sigma(t)}w(t)\\
     \tilde{y}(t)&=&\tilde{C}_{\sigma(t)}\tilde{x}(t)+\tilde{F}_{\sigma(t)}w(t)
\end{array}
\end{equation}
where $\sigma:\mathbb{R}_{\ge0}\mapsto\{1,\ldots,N\}$ is the switching signal, $\tilde x\in\mathbb{R}^n$ is the state of the system, $\tilde{w}\in\mathbb{R}^p$ is the exogenous input and $\tilde{y}\in\mathbb{R}^p$ is the measured output.

This system can be rewritten into the following impulsive system with multiple jump maps \cite{Briat:16c}
\begin{equation}\label{eq:swimp}
\begin{array}{rcl}
    \dot{x}(t)&=&\diag_{i=1}^N(\tilde{A}_{i})x(t)+\col_{i=1}^N(\tilde{E}_{i})w(t)\\
     y(t)&=&\diag_{i=1}^N(\tilde{C}_{i})x(t)+\col_{i=1}^N(\tilde{F}_{i})w(t)\\
     x(t_k^+)&=&J_{ij}x(t_k),\ i,j=1,\ldots,N,\ i\ne j
\end{array}
\end{equation}
where $J_{ij}:=(b_ib_j^T)\otimes I_n$ and $\{b_1,\ldots,b_N\}$ is the standard basis for $\mathbb{R}^N$.
Because of the particular structure of the system, we can define w.l.o.g.  an interval observer of the form \eqref{eq:obs} for the system \eqref{eq:swimp} with the gains $L_c(t)=\diag_{i=1}^N(L_c^i(t))$ and $L_d^{ij}=(b_ib_j^T)\otimes \tilde{L}_d^{ij}$. The error dynamics is then given in this case by
\begin{equation}\label{eq:error_switched}
\begin{array}{rcl}
    \dot{e}^\bullet(t)&=&\diag_{i=1}^N(\tilde{A}_{i}-L_c^i(t)\tilde{C}_i)e^\bullet(t)\\
    &&+\col_{i=1}^N(\tilde{E}_{i}-L_c^i(t)\tilde{F}_i)\delta^\bullet(t)\\
   e^\bullet(t_k^+)&=&\left[(b_ib_j^T)\otimes I_n-\tilde{L}_d^{ij}\tilde{C}_j)\right]e^\bullet(t_k)\\
   &&-\left[(b_ib_j^T)\otimes (\tilde{L}_d^{ij}\tilde{F}_j)\right]\delta^\bullet(t_k).
    \end{array}
\end{equation}
Once again, the gains of the observer can be designed as in the previous examples. Note, however, that in this case the stability conditions will be slightly different due to the existence of multiple jump maps. See \cite{Briat:16c} for more details on how to straightforwardly adapt the conditions to this case.

%\section{Conclusion}
%
%Several constructive conditions for the design of interval observers for linear impulsive systems have been obtained. The conditions can be checked using convex programming through the use of sum of squares methods or any alternative method based, for instance, on Handelman's Theorem. The versatility of impulsive systems allows us to also obtain for sampled-data and switched systems. Future works will be devoted to the derivation of additional observation conditions covering more dwell-time concepts, the use of such observers for stabilization and the consideration of performance measures such as the $L_1$-gain.

%
%\section{Conclusion}
%
%
%
%\begin{itemize}
%  \item Performance, Extension to uncertain systems
%\end{itemize}

\bibliographystyle{unsrt}
%\bibliography{../../../../Lastbib/global,../../../../Lastbib/briat}

\end{document}